\documentclass[12pt]{amsart}
\usepackage{amsmath,amssymb,amsfonts,a4wide,color,mathtools,url}
\usepackage[initials]{amsrefs}
\usepackage{cite}
\usepackage[utf8]{inputenc} 

\let\phi\varphi
\let\epsilon\varepsilon

\let\leq\leqslant
\let\geq\geqslant

\newtheorem{Thm}{Theorem}[section]
\newtheorem{Prop}[Thm]{Proposition}
\newtheorem{Lem}[Thm]{Lemma}
\newtheorem{Cor}[Thm]{Corollary}

\numberwithin{equation}{section}

\def\qed{{\hskip0pt\unskip\unskip\nobreak\hfil\penalty50
          \hskip1em\hbox{}\nobreak\hfil
           {$\square$}
          \parfillskip=0pt\finalhyphendemerits=0
          \par}\medskip}

               {\noindent{\bf Proof.}\ }
               {\qed}

               {\noindent{\bf Proof of #1.}\ }
               {\qed}

\begin{document}

\title{Computable Absolutely Pisot Normal Numbers}

\author[M. G. Madritsch]{Manfred G. Madritsch}
\address[M. G. Madritsch]{
\noindent 1. Universit\'e de Lorraine, Institut Elie Cartan de Lorraine, UMR 7502, Vandoeuvre-l\`es-Nancy, F-54506, France;\newline
\noindent 2. CNRS, Institut Elie Cartan de Lorraine, UMR 7502, Vandoeuvre-l\`es-Nancy, F-54506, France}
\email{manfred.madritsch@univ-lorraine.fr}

\author[A.-M. Scheerer]{Adrian-Maria Scheerer}
\address[A.-M. Scheerer]{Department for Analysis and Computational Number
  Theory\\Graz University of Technology\\A-8010 Graz, Austria}
\email{scheerer@math.tugraz.at}

\author[R. F. Tichy]{Robert F. Tichy}
\address[R. F. Tichy]{Department for Analysis and Computational Number
  Theory\\Graz University of Technology\\A-8010 Graz, Austria}
\email{tichy@tugraz.at}


\begin{abstract}
We analyze the convergence order of an algorithm producing the digits of an absolutely normal number. Furthermore, we introduce a stronger concept of absolute normality by allowing Pisot numbers as bases, which leads to expansions with non-integer bases.
\end{abstract}


\date{\today}

\maketitle

\section{Introduction}

In this paper we are interested in simultaneous normality to several bases. In particular, we analyze the order of convergence to normality of an absolutely normal number generated by an algorithm of Becher, Heiber and Slaman (Section~\ref{sec:discrepancy}) and are concerned with normality to real bases. We give an algorithmic construction of a real number that is normal to each base from a given sequence of Pisot numbers (Section~\ref{sec:abs_pisot} and Section~\ref{sec:ExplicitEstimates}).


\subsection{Normality to a single base} 

A real number $x\in[0,1)$ is called \emph{simply normal to base $b$}, $b\geq 2$ an integer, if in its $b$-ary expansion
\begin{equation*}
x=\displaystyle\sum_{n\geq 1}a_{n}b^{-n},\; a_{n}\in
\{0,\ldots , b-1\}
\end{equation*}
every digit $d \in \{0,1, \ldots b-1\}$ appears with the expected frequency, i.e. the limit 
\begin{equation*}
\lim_{n \to \infty} \frac{1}{N} \vert \{1\leq n \leq N : a_n = d\}\vert
\end{equation*}
exists and is equal to $\frac{1}{b}$.
The number $x$ is called \emph{normal to base $b$}, $b\geq 2$ an integer, if in its $b$-ary expansion all finite combinations of digits appear with the expected frequency, i.e. if for all $k\geq 1$ and all $d\in \{0,\ldots , b-1\}^{k}$,
\begin{equation}\label{limes_def_of_normality}
\lim_{N\rightarrow\infty}\frac{1}{N} \vert \{1 \leq n\leq N :(a_{n},\ldots , a_{n+k-1})=d \} \vert = \frac{1}{b^k}.
\end{equation}

Normal numbers were introduced by Borel~\cite{borel1909} in 1909. He showed that almost all real numbers (with respect to Lebesgue measure) are simply normal to all bases $b\geq2$, thus \emph{absolutely normal} (see Section~\ref{sec:abs_norm_and_order_of_convergence}).
It is a long standing open problem to show that important real
numbers such as $\sqrt{2}$,  $\ln 2, e, \pi, \ldots$ are normal, for
instance in decimal expansion. There has only been little progress in this direction in the last decades, see e.g. \cite{bailey2014:normality_of_constants}. 

However, specifically constructed examples of normal numbers are known.
Champernowne in 1935 \cite{champernowne} has shown that the real number constructed by concatenating the expansions in base $10$ of the positive integers, i.e.
\[0,1\,2\,3\,4\,5\,6\,7\,8\,9\,10\,11\dots,\]
is normal to base $10$. This construction has been extended in various directions (\textit{cf.} Erd{\H o}s and Davenport \cite{davenport_erdos:note_on_normal}, Schiffer
\cite{schiffer:discrepancy}, Nakai and Shiokawa \cite{nakai_shiokawa:class_of_normal_numbers}, Madritsch, Thuswaldner and Tichy \cite{madritsch_thuswaldner_tichy}, Scheerer \cite{scheerer2015:pisot}).

\subsection{Discrepancy of normal numbers} 
The \emph{discrepancy} of a sequence $(x_n)_{n\geq 1}$ of real numbers is defined as
\begin{equation*}
D_{N}(x_n)=\sup_{J} \left\vert \frac{1}{N} \vert \{1 \leq n\leq N: x_n \bmod 1 \in J\}\vert -\lambda(J)\right\vert,
\end{equation*}
where the supremum is extended over subintervals $J \subseteq [0,1)$ and where $\lambda$ denotes the Lebesgue measure. A sequence is \emph{uniformly distributed modulo $1$} if its discrepancy tends to zero as $N \rightarrow \infty$.

It is known \cite{wall1950:thesis} that $x$ is normal to base $b$ if and only if the sequence
$(b^{n}x)_{n\geq 1}$ is uniformly distributed modulo $1$.
Hence $x$ is normal to base $b$ if
and only
if $D_{N}(b^n x)\rightarrow 0$ as $N\rightarrow\infty$.
It is thus a natural quantitative measure for the normality of $x$ to base $b$ to consider the discrepancy of the sequence $(b^n x)_{n\geq 1}$.

Answering a question of Erd\H os,
in 1975 Philipp \cite{philipp1975:lacunary_limit}
has shown a law of the iterated logarithm for
discrepancies of lacunary sequences which implies $D_{N}(b^n x)=O(\sqrt{\log\log N / N})$ almost everywhere.
Recently, Fukuyama \cite{fukuyama2013:metric_discrepancy} was able to determine
\begin{equation*}
\limsup_{N\rightarrow\infty}\frac{D_{N}(b^nx)\sqrt{N}}{\sqrt{\log\log N}}= c (b)\quad\text{a.e.,}
\end{equation*}
for some explicit positive constant $c(b)$.
Schmidt \cite{schmidt1972:irreg_of_distr_7} showed that there is an absolute constant $c>0$ such that for any sequence $(x_n)_{n\geq 1}$ of real numbers $D_N(x_n) \geq c \frac{\log N}{N}$ holds for infinitely many $N$.
Schiffer \cite{schiffer:discrepancy} showed that the discrepancies of constructions of normal numbers in the spirit of Champernowne satisfy upper bounds of order $O(\frac{1}{\log N})$.
It is an open question whether there exist a real number $x$ and an integer $b \geq 2$ such that $D_N(b^n x) = O(\frac{\log N}{N})$.

\subsection{Absolute normality and order of convergence}
\label{sec:abs_norm_and_order_of_convergence}
A number $x$ is called \emph{absolutely normal} if it is normal to any integer base $b\geq 2$. Since normality to base $b$ is equivalent to simple normality to all bases $b^n$, $n\geq 1$, absolute normality is equivalent to simple normality to all bases $b\geq2$. 

Since most constructions of numbers normal to a single base $b$
are concatenations of the $b$-ary expansions of $f(n)$, $n\geq1$, where $f$ is a positive-integer-valued increasing function, they essentially depend on the choice of the base $b$. Therefore they cannot be used for producing absolutely normal numbers.

All known examples of absolutely normal numbers have been established in the form of algorithms\footnote{With the exception of Chaitin's constant, which is absolutely normal but not computable~\cite{chaitin}.} that output the digits of this number to some base one after the other.
The first such construction is due to Sierpinski \cite{sierpinski1917:borel_elementaire} from 1917. This construction was made computable by Becher and Figueira \cite{becher_figueira:2002} who gave
a recursive formulation of Sierpinski's construction.
Other algorithms for constructing absolutely normal numbers are due to Turing \cite{turing1992:collected_works} (see also Becher, Figueira and Picchi~\cite{becher_figueira_picchi2007:turing_unpublished}), Schmidt~\cite{schmidt1961:uber_die_normalitat} (see also Scheerer~\cite{scheerer2015:schmidt}) and Levin~\cite{levin1979:absolutely_normal} (see also Alvarez and Becher~\cite{alvarez_becher2015:levin}).



There seems to be a trade-off between the complexity of the algorithms and the speed of convergence of the corresponding discrepancies. 
The discrepancies satisfy upper bounds of the order $O(N^{-1/6})$ (Sierpinski), $O(N^{-1/16})$ (Turing), $O((\log N)^{-1})$ (Schmidt) and \allowbreak $O(N^{-1/2} (\log N)^3)$ (Levin). All algorithms, except the one due to Schmidt, need double exponential many mathematical operations to output the first $N$ digits of the produced absolutely normal number. Schmidt's algorithm requires exponentially many mathematical operations.

No construction of an absolutely normal number $x$ is known such that the discrepancy $D_N(b^n x)$ for some $b\geq 2$ decays faster than what one would expect for almost all $x$.\\

\noindent 
In Section~\ref{sec:discrepancy} we are interested in another construction of an absolutely normal number which is due to Becher, Heiber and Slaman \cite{becher_heiber_slaman2013:polynomial_time_algorithm}. They established an algorithm which computes the digits of an absolutely normal number in polynomial time. We show (Theorem~\ref{D_N_BHS}) that the corresponding discrepancy is  slightly worse than $O(\frac{1}{\log N})$, and that at a small loss of computational speed the discrepancy can in fact be $O(\frac{1}{\log N})$.

\subsection{Normality to non-integer bases}
Section~\ref{sec:abs_pisot} of the present article treats normality in a context where the underlying base is not necessarily integer. Let $\beta>1$ be a real number. 
Expansions of real numbers to base $\beta$, so-called \emph{$\beta$-expansions}, were introduced and studied by R\'{e}nyi \cite{renyi1957:reps_for_real_numbers}
and Parry \cite{parry1960:beta_expansions} and later by many authors from an arithmetic and ergodic-theoretic point of view. 

In the theory of $\beta$-expansions it is natural to consider \emph{Pisot numbers} $\beta$, i.e. real algebraic integers $\beta>1,$ such that all its conjugates lie inside the (open) unit disc. A real number $x$ is called \emph{normal to base $\beta$}, or $\beta$-normal, if the sequence
$\left(\beta^{n} x\right)_{n\geq 1}$ is uniformly distributed
modulo 1 with respect to the unique entropy maximizing measure for the underlying transformation $x \mapsto \beta x \bmod 1$ (see Section~\ref{Sec:defbetaexpansions}). A real number is called \textit{absolutely Pisot normal} if it is normal to all bases that are Pisot numbers. Since there are only countably many Pisot numbers, the Birkhoff ergodic theorem implies that almost all real numbers are in fact absolutely Pisot normal. \\


\noindent The main result of Section~\ref{sec:abs_pisot} is an algorithm that computes an absolutely Pisot normal number. More generally, for a sequence $(\beta_j)_{j\geq 1}$ of Pisot numbers, we construct a real number $x$ that is normal to each of the bases $\beta_j$, $j \geq 1$ (Section~\ref{sec:beta_algorithm} and Theorem~\ref{thm_beta_normal_algorithm}). 
Bearing in mind that the set of computable real numbers is countable, we thus show that there is in fact a computable real number that is $\beta_j$-normal for each $j \geq 1$.

Our algorithm constructs in each step a sequence of finitely many nested intervals, corresponding to the first finitely many bases considered. This is also the essential idea of the construction of an absolutely normal number by Becher, Heiber and Slaman \cite{becher_heiber_slaman2013:polynomial_time_algorithm}.
We need to establish lower and upper bounds for the length of $\beta$-adic subintervals in a given interval to control the number of specified digits when changing the base. 
However, the equivalence (absolute normality) $\Leftrightarrow$ (simple normality to all bases) does not hold for non-integer expansions. Instead, we argue with the concept of \emph{$(\varepsilon,k)$-normality} as introduced by Besicovitch \cite{besicovitch1935:epsilon} and studied in the case of Pisot numbers by Bertrand-Mathis and Volkmann \cite{bertrand-mathis_volkmann1989:epsilon_k_normal}.

Our algorithm should be compared to the one due to Levin~\cite{levin1979:absolutely_normal}. While his construction is not restricted to Pisot numbers, it uses exponential sums and is as such not realizable only with elementary operations. The algorithm we present in Section~\ref{sec:abs_pisot} is completely elementary.\\

\noindent In Section~\ref{sec:ExplicitEstimates} we give explicit estimates of all constants that appear in our algorithm. We use a theorem on large deviations for a sum of dependent random variables to give an estimate for the measure of the set of non-$(\epsilon,k)$-normal numbers of length $n$ (Proposition~\ref{non_normal_expl_constants}). Our approach gives all implied constants explicitly, and as such makes a consequence of the ineffective Shannon-McMillan-Breimann theorem effective. The results of this section might be of independent interest.

\subsection{Notation} 
For a real number $x$, we denote by $\lfloor x \rfloor$ the largest integer not exceeding $x$. The fractional part of $x$ is denoted as $\{x\}$, hence $x = \lfloor x \rfloor + \{x\}$. We put $\lceil x \rceil = - \lfloor -x \rfloor$. 
Two functions $f$ and $g$ are $f = O(g)$ or equivalently $f \ll g$ if there is a $x_0$ and a positive constant $C$ such that $f(x) \leq C g(x)$ for all $x\geq x_0$. We mean $\lim_{x\to\infty} f(x)/g(x) = 1$ when we say $f\sim g$ and $g\neq 0$.

When we speak of \emph{words}, we mean finite or infinite sequences of symbols (called letters) of a certain (specified) set, the alphabet. \emph{Blocks} are finite words. The concatenation of two blocks $u = u_1 \ldots u_k$ and $v_1 \ldots v_l$ is the block $u_1 \ldots u_k v_1 \ldots v_l$ and is denoted by $uv$ or $u\ast v$. If $u_i$ for $i\leq m$ are blocks, $\ast_{i<m} u_i$ is their concatenation in increasing order of $i$. The \emph{length} of the block $u = u_1 \ldots u_k$ is denoted by $\Vert u \Vert$ and is in this case equal to $k$.

We denote by $\lambda$ the Lebesgue measure. 

For a finite set, $\vert \cdot \vert$ means its number of elements.


\emph{Mathematical operations} include addition, subtraction, multiplication, division, comparison, exponentiation and logarithm. \emph{Elementary operations} take a fixed amount of time. The cost of mathematical operations depends on the digits of the input or on the desired precision of the output. Addition or subtraction of two $n$-digit numbers takes $O(n)$ elementary operations, multiplication or division of two $n$-digit numbers takes $O(n^2)$ elementary operations, and to compute the first $n$ digits of $\exp$ and $\log$ takes $O(n^{5/2})$ elementary operations. These estimates are crude but sufficient for our purposes. 

The complexity of a computable function $f$ is the time it takes to compute the first $N$ values $f(i)$, $1\leq i \leq N$. The algorithm we analyze 
outputs the digits of a real number $X$ to some base. By the complexity of the algorithm we mean the time it takes to output the first $N$ digits of $X$ to some base.

\section{Discrepancy}
\label{sec:discrepancy}

\noindent In this section, we analyze the speed of convergence to normality of the absolutely normal number produced by the algorithm by Becher, Heiber and Slaman in \cite{becher_heiber_slaman2013:polynomial_time_algorithm}. We follow the notation and terminology therein.

\subsection{The Algorithm}

\subsubsection*{Notation}
A \emph{$t$-sequence} is a nested sequence of intervals $\mathbf{I} = (I_2, \ldots, I_t)$, such that $I_2$ is dyadic and for each base $2\leq b \leq t-1$, $I_{b+1}$ is a $(b+1)$-adic subinterval of $I_b$ such that $\lambda(I_{b+1}) \geq \lambda(I_b)/2(b+1)$.

Let $x_b(\mathbf{I})$ be the block in base $b$ such that $0.x_b(\mathbf{I})$ is the representation of the left endpoint of $I_b$ in base $b$. In each step $i$, the algorithm computes a sequence $\mathbf{I}_i = (I_{i,2}, \ldots, I_{i,t_i})$ of nested intervals $I_{i,2} \supset \ldots \supset I_{i, t_i}$. If $b\leq t_i$, let $x_b(\mathbf{I}_i) = x_{i,b}$ be the base $b$ representation of the left endpoint of $I_{i,b}$ and let $u_{i+1,b} = u_b(\mathbf{I}_{i+1})$ be such that $x_{i+1,b} = x_{i,b} \ast u_{i+1,b}$.

If $u$ is a block of digits to base $b$, the \emph{simple discrepancy} of $u$ in base $b$ is defined as $D(u,b) = \max_{0\leq d < b} \vert N_d(u)/\Vert u \Vert - 1/b \vert$ where $N_d(u)$ is the number of times the digit $d$ appears in the block $u$.

Let $k(\epsilon, \delta, t)$ be the function
\begin{equation*}
k(\epsilon, \delta, t) = \max(\lceil 6 / \epsilon \rceil, \lceil - \log(\delta/(2t)) 6 / \epsilon^2 \rceil ) +1.
\end{equation*}

From Lemma 4.1 and 4.2 of \cite{becher_heiber_slaman2013:polynomial_time_algorithm} we further have a function $h$
that counts the number of mathematical operations needed to carry out one step of the algorithm. See also Lemma \ref{growth_h_i}.


\subsubsection*{Input} A computable non-decreasing unbounded function $f : \mathbb{N} \rightarrow \mathbb{R}$ such that $f(1)$ is known and satisfies $f(1) > h(2,1)$.

\subsubsection*{First step}

Set $t_1 = 2$, $\epsilon_1 = \frac{1}{2}$, $k_1 = 1$ and $\mathbf{I}_1 = (I_{1,2})$ with $I_{1,2} = [0,1)$.

\subsubsection*{Step $i+1$ for $i \geq 1$}
Given are from step $i$ of the algorithm values $t_i = v$, $\epsilon_i = \frac{1}{v}$ and a $t_{i}$-sequence $\mathbf{I}_{i}$.

We want to assign values to $t_{i+1}, \epsilon_{i+1}$. If $i+1$ is a power of $2$, then we carry out the following procedure.
\begin{itemize}

\item We spend $i$ computational steps on computing the first $m$ values of $f$, $1 \leq m \leq i$.
\item We put $\delta = 	(8 t_i 2^{t_i + v+1} t_i! (v+1)!)^{-1}$.
\item We try to compute $k(\frac{1}{v+1}, \delta, v+1)$ and $h(v+1, \frac{1}{v+1})$ in $i$ steps each. If we succeed in computing these values, and if additionally
\begin{equation}\label{cond_h}
h(v+1, \frac{1}{v+1}) < f(m)
\end{equation}
and for each $b\leq t_i$
\begin{equation}\label{cond_k}
\frac{\lceil \log_2(v+1) \rceil k(1/(v+1), \delta, v+1) + \lceil - \log_2(\delta) \rceil }{\Vert x_{i,b} \Vert} < \frac{1}{v+1},
\end{equation}
then we define $t_{i+1} = v+1$ and  $\epsilon_{i+1} = \frac{1}{v+1}$.
Otherwise, we let $t_{i+1} = t_i = v$, $\epsilon_{i+1} = \epsilon_i = \frac{1}{v}$.
\end{itemize}
If $i+1$ is no power of $2$, then define $t_{i+1} = t_i = v$, $\epsilon_{i+1} = \epsilon_i = \frac{1}{v}$.

Furthermore, we compute $\delta_{i+1} = (8 t_i 2^{t_i + t_{i+1}} t_i! t_{i+1}!)^{-1}$ and 
\begin{equation*}
k_{i+1}= \max(\lceil 6 / \epsilon_{i+1} \rceil, \lceil - \log(\delta_{i+1} /(2t_i)) 6 / \epsilon_{i+1}^2 \rceil ) +1.
\end{equation*}

Then we find a $t_{i+1}$-sequence $\mathbf{I}_{i+1}$ by means of the following steps.
\begin{itemize}
 \item We let $L$ be a dyadic subinterval of $I_{i,t_i}$ such that $\lambda(L) \geq \lambda(I_{i,t_i})/4$.
 \item For each dyadic subinterval $J_2$ of $L$ of measure $2^{-\lceil \log_2 t_i \rceil k_{i+1}} \lambda(L)$, we find $\mathbf{J} = (J_2, J_3, \ldots, J_{t_{i+1}})$, a $t_{i+1}$-sequence starting with $J_2$.
 \item Finally we choose $\mathbf{I}_{i+1}$ to be the leftmost of the $t_{i+1}$ sequences $\mathbf{J}$ considered above such that for each $b\leq t_i$, $D(u_b(\mathbf{J}),b) \leq \epsilon_{i+1}$.
\end{itemize}


\subsubsection*{Output}
Let $X$ be the unique real number in the intersection of the intervals of the sequences $\mathbf{I}_i$.
In base $b$ we have $X = \lim_{i \rightarrow \infty} 0.x_{i,b} = 0.\ast_{i \geq 1} u_{i,b}$. 
It is the content of Theorem 3.9 in \cite{becher_heiber_slaman2013:polynomial_time_algorithm} that $X$ is absolutely normal.

\subsection{Speed of convergence to normality}

In this section we estimate the discrepancy $D_N(b^n X)$ for integer $b \geq 2$.
Two factors play a role: How many digits in each step are computed, and how rapidly $\epsilon_i$ decays to zero. By virtue of the algorithm, at least one digit is added in each step, and $\epsilon_i$ can decay at most as fast  as $O(\frac{1}{\log i})$. As can be expected from the algorithm, the discrepancy depends both on growth and  complexity of $f$.

It was shown in~\cite{becher_heiber_slaman2013:polynomial_time_algorithm} that to output the first $N$ digits of $X$, the algorithm requires time $O(N^2 f(N))$.\\

\noindent We begin our analysis by first showing that in each step of the algorithm not too many digits are attached.

\begin{Lem}[Lemma 3.3 in \cite{becher_heiber_slaman2013:polynomial_time_algorithm}] \label{lowerboundintervals}
For an interval $I$ and a base $b$, there is a $b$-adic subinterval $I_b$ such that $\lambda(I_b) \geq \lambda(I)/(2b)$.
\end{Lem}

\begin{Lem}\label{length_of_blocks_in_BHS}
If $i$ is large enough, then $ 1 \leq \Vert u_{i,b} \Vert \ll (\log i)^A$ for $A>3$. Thus $i \ll \Vert x_{i,b} \Vert \ll i (\log i)^A$.
\end{Lem}

\begin{proof}
We assume the base $b$ to be fixed and $i$ large enough such that $t_{i+1} \geq b$. In step $i+1$ we have the following sequence of nested subintervals:
\begin{equation}\label{nested}
I_{i,b} \supset \ldots \supset I_{i, t_i} \supset L \supset I_{i+1,2} \supset \ldots \supset I_{i+1,b}.
\end{equation}
By Lemma \ref{lowerboundintervals}, and the choice of $I_{i+1,2}$, we know the following lower bounds on the measures of the intervals in \eqref{nested}. We have $\lambda(I_{i,t_i}) \geq \lambda(I_{i,b}) / (2^{t_i-b} t_i!/b!)$, $\lambda(L) \geq \lambda(I_{i,t_i})/4$, $\lambda(I_{i+1,2}) = 2^{-\lceil \log_2 t_i \rceil k_{i+1}} \lambda(L)$ and $\lambda(I_{i+1,b}) \geq \lambda(I_{i+1,2}) / (2^b b!)$. Combining inequalities yields $\lambda(I_{i+1,b}) \geq \lambda(I_{i,b}) / ( 2^{2+t_i} 2^{\lceil \log_2 t_i \rceil k_{i+1}} t_i!)$. Hence in stage $i+1$ we are adding at most $O(t_i + (\log t_i) k_{i+1} + \log t_i!)$ many digits in base $b$. The way the algorithm is designed only allows for $t_i = O(\log i)$. The growth of $k_{i+1}$ can be analyzed and is $O(t_i^3 \log t_i)$. Hence in stage $i+1$ at most $O(t_i + (\log t_i) k_{i+1} + \log t_i!) = O((\log i)^A)$ digits are added to the $b$-ary expansion of $X$, where $A>3$ to accommodate all double-log factors. 

The lower bound on the number of digits added comes from the fact that by the choice of $I_{i+1,2}$, $I_{i+1,b}$ is strictly smaller than $I_{i,b}$, so at least one digit is added in each stage.
\end{proof}

\noindent Next, we investigate the conditions involving $k$ and $h$ that are responsible for how fast $t_i \rightarrow \infty$ and $\epsilon_i \rightarrow 0$ with step $i$ of the algorithm. We start by showing that condition \eqref{cond_k} on $k$ always holds, provided $i$ is large enough. This involves estimating the growth as well as the complexity of $k$. 

Recall that $k(\epsilon, \delta, t) = \max(\lceil 6 / \epsilon \rceil, \lceil - \log(\delta/(2t)) 6 / \epsilon^2 \rceil ) +1$.

\begin{Lem}
Let $v \geq 2$ be an integer and $\delta =   (8 v 2^{2v+1} v! (v+1)!)^{-1}$. Then the growth of $k(\frac{1}{v+1},\delta, v+1)$ is $O(v^3 \log v)$. Furthermore, $k(\frac{1}{v+1},\delta, v+1)$ can be computed in $O(v^2(\log v)^2)$ elementary operations.
\end{Lem}

\begin{proof}
We have for the growth
\begin{align*}
k(\frac{1}{v+1}, \delta, v+1) &= \max(\lceil 6(v+1) \rceil, \lceil \log(2(v+1) 8v 2^{2v+1} v! (v+1)!)6(v+1)^2\rceil) + 1\\
&\leq 6(v+1)^2 \left( \log(16v(v+1)) + (2v+1) \log 2 + \log v! + \log(v+1)! \right) + 2 \\
&= O(v^2(\log v + v + v \log v)) \\
&= O (v^3 \log v).
\end{align*}
Since in the expression for $k$ we are rounding, the most relevant part is the computation of the significant digits of $\log(16v(v+1) 2^{2v+1} v! (v+1)!)$. The argument of this expression is computable with $O(v^2(\log v)^2)$ elementary operations and has $O(v\log v)$ many digits. We only need to compute $O(\log v)$ many digits of the logarithm, which takes another $O((\log v)^{5/2})$ elementary operations. In total this are $O(v^2(\log v)^2)$ many elementary operations.
\end{proof}

\begin{Cor}
For $i$ to be large enough, condition~\eqref{cond_k} on $k$ is always satisfied, i.e. for each $b\leq t_i$
\begin{equation*}
\frac{\lceil \log(v+1) \rceil k(1/(v+1), \delta, v+1) + \lceil - \log(\delta) \rceil }{\Vert x_{i,b} \Vert} < \frac{1}{v+1}
\end{equation*}
where $v$ is such that $t_i = v = 1/\epsilon_i$.
\end{Cor}

\begin{proof}
This is a consequence of $k(1/(v+1), \delta, v+1) = O(v^3 \log v)$, $\log(1/\delta) = O(v\log v)$, $\Vert x_{i,b}\Vert \gg i$ and $v= t_i = O(\log i)$ by the way the algorithm is designed.
\end{proof}


\noindent Now we investigate condition \eqref{cond_h} on $h$ involving $f$. The function $h$ counts the number of mathematical operations needed to carry out one step of the algorithm. We want to know an upper bound for the growth of $h$.




%

\begin{Lem}\label{growth_h_i}
With $t_i = \frac{1}{\epsilon_i} = O(\log i)$ we have
\begin{equation*}
h(t_i, \epsilon_i) = O(i^{\log^4 i}).
\end{equation*}
This upper bound for $h$ can be computed with $i$ elementary operations, provided $i$ is large enough.
\end{Lem}

\begin{proof}
The function $h$ decomposes as $h = h_\ast ( h_1 g + h_2 + h_3 + h_4) h_0$ as can be seen from the proof of Lemma 4.2 in \cite{becher_heiber_slaman2013:polynomial_time_algorithm}.
Here: 
\begin{itemize} 
\item  $g$ (from Lemma 4.1 in \cite{becher_heiber_slaman2013:polynomial_time_algorithm}), is the minimum number of digits sufficient to represent all the endpoints of the intervals that we are working with in one step (squared). We know from Lemma~\ref{length_of_blocks_in_BHS} that $g = O( i^2 (\log i)^{2A})$ for $A>3$. 

\item It takes $h_1 g$ many mathematical operations to find a $t_{i+1}$-sequence for each $J_2$. We have $h_1 = t_{i+1}$.

\item $h_2$ is the number of mathematical operations needed to compute the base $b$ representation $u_b(\mathbf{J})$ for each $2 \leq b \leq t_i$. We have $h_2 \leq \lceil \log_2 t_i \rceil k_{i+1}$.

\item $h_3$ counts the number of mathematical operations needed to compute thresholds of the form $(1/b + \epsilon_{i+1}) \Vert u_b(\mathbf{J}) \Vert$. We have $h_3 = t_i$.

\item $h_4$ comes from counting occurrences of digits in $u_b(\mathbf{J})$ and comparing with the previously computed thresholds. We have $h_4 \ll t_i (\lceil \log_2 t_i \rceil k_{i+1})^2$. 

\item $h_\ast$ is the maximum number of iterations it takes to find a suitable $t_{i+1}$-sequence. There are $2^{\lceil \log_2 t_i\rceil k_{i+1}}$ many different subintervals $J_2$ of $L$, hence $h_\ast = 2^{\lceil \log_2 t_i\rceil k_{i+1}}$. With $k_{i+1} = O(\log^4 i)$ we obtain $h_\ast = O(i^{\log^4 i})$. 

\item Finally, the function $h_0$ is the number of elementary operations needed to carry out each mathematical operation in one step of the algorithm. Since all values that appear in the calculations of one step of the algorithm are at most exponential in $t_i$ which is at most of order $\log i$, and because the number of elementary operations involved depends only on the number of digits of the numbers involved, $h_0$ is at most of order $\text{poly}(\log i)$.
\end{itemize}

These bounds can be seen from Lemma~4.1 and Lemma~4.2 in~\cite{becher_heiber_slaman2013:polynomial_time_algorithm}. Combining them gives $h = O(i^{\log^4 i})$.

Remark that, when $t_i$ is bounded by a slower growing function in $i$ such as $\log \log i$, then the significant term in $h$ comes from $g$ and is a power of $i$. Otherwise $h_\ast$ is the significant term.

For the complexity of the upper bound for $h$, note that $i^{\log^4 i}$ can be computed in a power of $\log i$ many elementary operations, so certainly with $i$ elementary operations when $i$ is large enough.
\end{proof}

\noindent Lemma \ref{growth_h_i} has the following two immediate corollaries for the speed of convergence to normality of Becher, Heiber and Slaman's algorithm.

\begin{Prop}
Becher, Heiber, Slaman's algorithm achieves discrepancy of $D_N(b^n X) = O(\frac{1}{\log N})$ for $f$ computable in real-time with growth $f \gg i^{\log^4 i}$. In this case, the complexity is $O(i^{2+ \log^4 i})$.
\end{Prop}

\begin{Prop}
If $f$ is a polynomial in $i$ of degree $d$, then the complexity of $X$ is $O(N^{d+2})$ but the discrepancy of $(b^n X)_{n\geq 0}$ is $D_N(b^nX) = O_d(\frac{1}{(\log N)^{1/5}})$.
\end{Prop}

\begin{proof}
These corollaries follow by observing that the complexity of $f$ is such that $f$ is for large enough $i$ computed up to the actual value $f(i)$ (i.e. $m=i$) and that either the condition on $h$, \eqref{cond_h}, is satisfied, hence the discrepancy is optimal, or that condition \eqref{cond_h} is only satisfied for $e^{(\log i)^{1/5}}$ of the values that it is checked for.
\end{proof}

\noindent In a similar manner, using Lemma \ref{growth_h_i}, one can show quantitatively how growth and complexity of $f$ influence the discrepancy (and the complexity) of Becher, Heiber, Slaman's algorithm. This can be done for example by measuring complexity and growth of $f$ in the following (crude) way. We denote by $\log_{(k)}$ and $\exp_{(k)}$ the $k$ times iterated logarithm or exponential where $\exp_{(k)} = \log_{(-k)}$, and $\exp_{(0)} = \log_{(0)} = \mathit{id}$. Let $c$ be the integer such that in $i$ elementary operations $f$ can be computed up to a value $f(m)$ with $m \sim \log_{(c)} i$. Let $g$ be the integer such that $f$ grows as $f \sim \exp_{(g)} i$. We allow $g\in \mathbb{Z}$ but $c$ is non-negative.

\begin{Thm}\label{D_N_BHS}
Assume $f$ is such that the integers $c$ and $g$ above can be defined.  Then Becher, Heiber, Slaman's algorithm computes an absolutely normal number $X$ such that for any base $b\geq 2$,  
\begin{equation}
D_N(b^n X) = O\left(\frac{1}{(\log_{(1-g+c)} N)^{1/5}}\right)
\end{equation}
if $1-g+c > 0$, and
\begin{equation}
D_N(b^n X) = O\left(\frac{1}{\log N}\right)
\end{equation}
otherwise.
\end{Thm}

\begin{proof}
We have $h \ll \max( \text{poly}(i), e^{t_i^5})$ and $t_i \ll \log i$ by the way the algorithm is defined. $t_i$ only increases if $i$ is a power of two and if $h\leq f(m)$. The latter condition is satisfied for all $i$ large enough if $g-c \geq 1$, and for all $i$ (that are powers of two) that satisfy $i \ll \exp((\exp_{g-c-1}(i))^{1/5})$. With $1/t_i = \epsilon_i$ this gives in this case an upper bound for the discrepancy of order $1/(\log_{(1-g+c)} N)^{1/5}$.
\end{proof}

\section{Absolutely Pisot Normal Numbers}\label{sec:abs_pisot}

\noindent In this section, we give an algorithmic construction of a real number that is normal to each base from a given sequence of Pisot numbers.
For more information about $\beta$-expansions and $\beta$-normal numbers see for example the book \cite{bugeaud2012distribution}. 
We have partly followed the notation in \cite{bertrand-mathis_volkmann1989:epsilon_k_normal}.

\subsection{$\beta$-expansions of real numbers}
\label{Sec:defbetaexpansions}

Let $\beta >1$ be a real number. Then each real number $x\in [0,1)$ has a representation of the form
\begin{equation}\label{beta-exp}
x = \sum_{i=1}^{\infty} \epsilon_i \beta^{-i},
\end{equation}
with integer digits $0 \leq \epsilon_i < \beta$. 
One way to obtain such a representation is the following.
Let $T_\beta$ be the \emph{$\beta$-transformation} $T_\beta : [0,1) \rightarrow [0,1)$, \mbox{$x \mapsto \beta x \pmod 1$.}
Then $\epsilon_i = \lfloor \beta T_\beta^{i-1}(x) \rfloor$ for $i\geq 1$.

R\'{e}nyi \cite{renyi1957:reps_for_real_numbers} showed that there is a unique probability measure $\mu_\beta$ on $[0,1)$ that is equivalent to the Lebesgue measure and such that $\mu_\beta$ is invariant and ergodic with respect to $T_\beta$ and has maximum entropy. 
The measure $\mu_\beta$ satisfies $(1-\frac{1}{\beta})\lambda \leq \mu_\beta \leq \frac{\beta}{\beta-1}\lambda$.

Let $c(d)$ be the \emph{cylinder set} corresponding to the block $d$, i.e. the set of all real numbers in the unit interval whose first $\Vert d \Vert$ digits coincide with $d$. A \emph{$\beta$-adic interval} is a cylinder set $c(d)$ for some $d$.



Let $W^\infty$ be the set of right-infinite words $\omega = \omega_1 \omega_2 \ldots$ with digits $0 \leq \omega_i < \beta$ that appear as the $\beta$-expansions of real numbers in the unit interval. Let $\mathcal{L}_n$ be the set of all finite subwords of length $n$ of words $\omega \in W^\infty$ and let $W = \bigcup_{n\geq 1} \mathcal{L}_n$. We call the words in $W$ \emph{admissible}.

We have $\beta^n \leq \vert \mathcal{L}_n \vert \leq \frac{\beta}{\beta-1} \beta^n$ for the number of elements of $\mathcal{L}_n$.

For an infinite word $\omega = \omega_1 \omega_2 \ldots \in W^\infty$ and a block $d=d_1d_2 \ldots d_k$ of digits $0\leq d_i < \beta$ we denote by $N_d(\omega,n)$ the number of (possibly overlapping) occurrences of $d$ within the first $n$ letters of $\omega$. If the word $\omega$ is finite, we write $N_d(\omega)$ for $N_d(\omega, \Vert \omega \Vert)$.

 An infinite word $\omega \in W^\infty$ is called \emph{$\mu_\beta$-normal} if 
for all $d\in \mathcal{L}_k$,
\begin{equation*}
\lim_{n\rightarrow \infty} \frac{1}{n} N_d(\omega,n) = \mu_\beta(c(d)).
\end{equation*}
A real number $x \in [0,1)$ is called \emph{normal to base $\beta$} or \emph{$\beta$-normal}, if the infinite word $\epsilon_1 \epsilon_2 \ldots$ defined by its $\beta$-expansion~\eqref{beta-exp} is $\mu_\beta$-normal.

For fixed $\epsilon>0$ and positive integers $k$, $n$, a word $\omega \in \mathcal{L}_n$ is called \emph{$(\epsilon,k)$-normal} if for all $d \in \mathcal{L}_k$
\begin{equation*}
\mu_\beta(c(d)) (1-\epsilon) \Vert \omega \Vert< N_d(\omega)  < \mu_\beta(c(d)) (1+\epsilon) \Vert \omega \Vert.
\end{equation*}
The set of all $(\epsilon,k)$-normal numbers in $\mathcal{L}_n$ will be denoted by $E_n(\epsilon,k)$ and its complement by $E^c_n(\epsilon,k)$.

A \emph{Pisot number} $\beta$ is a real algebraic integer $\beta > 1$ such that all its conjugates have absolute value less than $1$, and as usual we include all positive integers $b\geq 2$ in this definition. All Pisot numbers smaller than the golden mean were found by Dufresnoy and Pisot \cite{dufresnoy_pisot1955:etude_de_certaines}. In particular, they showed that the smallest one is the positive root of $x^3-x-1$ (called the plastic number) which is approximately $1.32471>\sqrt[3]{2}$.


\subsection{Preliminaries}

\begin{Lem}[{\cite[Lemma 3]{bertrand-mathis_volkmann1989:epsilon_k_normal}}]\label{bv:lem3}
Let $\beta>1$ be Pisot. For every $\varepsilon>0$ and positive integer $k$ there exist $\eta=\eta(\varepsilon,k)$, $0 < \eta < 1$, 
$C=C(\varepsilon,k)>0$ and $n_0 = n_0(\epsilon, k)$ such that for the number of non-$(\epsilon,k)$-normal words of length $n$
\begin{equation*}
\vert E^c_n(\epsilon,k) \vert <C\left| \mathcal{L}_n\right|^{1-\eta}
\end{equation*}
holds for all $n\geq n_0$.
\end{Lem}

\noindent In Section~\ref{large_deviation_section} we give explicit estimates for $n_0$, $C$ and $\eta$.\\


\noindent The following Lemma contains the underlying idea of our construction.


\begin{Lem}[{\cite[Lemma 4]{bertrand-mathis_volkmann1989:epsilon_k_normal}}]\label{bv:lem4}
Let $a_1,a_2,\ldots$ be a sequence of finite words $a_n\in W$ such that $a=a_1 a_2 \ldots \in W^\infty$ and
$\Vert a_n \Vert\to\infty$ as $n\rightarrow \infty$. Suppose that for any $\epsilon>0$ and any positive integer $k$ there exists an integer $n_0(\epsilon,k)$ such that all $a_n$ with $n\geq n_0(\epsilon,k)$ are
$(\epsilon,k)$-normal.
If
\begin{equation}\label{growth_conditions_for_normality}
n=o\left(\Vert a_1 a_2 \ldots a_n \Vert  \right)
\quad \text{and} \quad
\Vert a_{n+1} \Vert = o(\Vert a_1 a_2 \ldots a_n \Vert),
\end{equation}
then the infinite word $a=a_1 a_2 \ldots$ is $\mu_\beta$-normal.
\end{Lem}

\begin{proof}
Let $\epsilon>0$ and $d \in \mathcal{L}_k$.
It suffices to show that, as $N\rightarrow\infty$,
\begin{equation*}
\mu_\beta(c(d))(1-\epsilon)N < N_d(a,N) < \mu_\beta(c(d))(1+\epsilon)N.
\end{equation*}
We have $N_d(a,N) = N_d(a_1a_2 \ldots a_n, N)$, where $n$ is such that $\Vert a_1 a_2 \ldots a_{n-1} \Vert < N \leq \Vert a_1 \ldots a_n \Vert$. Then, for $N$ large enough,
\begin{align*}
N_d(a_1 \ldots a_n, N) & \leq N_d(a_1 \ldots a_{n_0(\epsilon,k)}) + n(k-1) + N_d(a_{n_0+1}) +  \ldots + N_d(a_n) \\
&\leq \text{const$(\epsilon,k)$} + n(k-1) + \sum_{i=n_0+1}^n \mu_\beta(c(d))(1+\epsilon)\Vert a_i \Vert.
\end{align*}
Dividing by $N$ gives the desired result, assuming conditions~\eqref{growth_conditions_for_normality}. The calculation for the lower bound for $N_d(a,N)$ is similar.
\end{proof}


\begin{Lem}\label{liwu:prop2.6}
Let $\beta>1$ be Pisot. There exists $M\geq 0$ such that for all $n\geq1$ and all $d\in L_n$  the Lebesgue measure of the cylinder set $c(d)$ satisfies
\begin{equation}\label{lower_bound_measure_beta_adic_interval}
\beta^{-(M+1)} \beta^{-n}\leq \lambda(c(d)) \leq \beta^{-n}.
\end{equation}
\end{Lem}

\begin{proof}
This is Proposition 2.6 of \cite{li_wu2008:beta_expansion_and}.
\end{proof}

\noindent Following the argument in \cite{li_wu2008:beta_expansion_and}, one can take $M$ to be the size of the largest block of consecutive zeros in the modified $\beta$-expansion of $1$ (see Section~\ref{subsec:number_of_zeros}). We give an explicit upper bound on $M$ in Proposition~\ref{Proposition_on_M}.\\

\noindent We wish to control the lengths when changing the base. The following is an analogue to Lemma~3.3 in \cite{becher_heiber_slaman2013:polynomial_time_algorithm}; see also Lemma~\ref{lowerboundintervals}.

\begin{Lem}\label{inscribed_interval_beta}
Let $\beta$ be Pisot and $M$ as above. For
any interval $I$ there is a $\beta$-adic subinterval $I_\beta$ of $I$ such that
$\lambda(I_\beta) \geq \lambda(I) / 2\beta^{M+4}.$
\end{Lem}

\begin{proof}
We can assume $\lambda(I) > 0$.
Let $m$ be the smallest integer such that $\beta^{-m} < \lambda(I)$. Thus $\lambda(I) / \beta\leq\beta^{-m}<\lambda(I)$. If there exists an
interval of order $m$ in $I$, then let $I_\beta$ be this $\beta$-adic
interval and we have $\lambda(I_\beta) \geq \lambda(I) /\beta$.

Otherwise there must be a word $a\in \mathcal{L}_m$ such that $\pi(a)\in I$ but
neither $\pi(a^-)$ nor $\pi(a^+)$ is in $I$, where $a^-$ and $a^+$ are the lexicographically previous or next elements of $a$ of the same length and where $\pi(a)$ is the real number in the unit interval whose $\beta$-expansion starts with $a$. Then by
Lemma~\ref{liwu:prop2.6} we have that $\lambda(I) < 2\beta^{-m}$.
Since $\beta^{-m}<\lambda(I)$ and the smallest Pisot number is
bigger than $2^{1/3}$, we get that $2\beta^{-m-3}<\lambda(I)$.
Thus there must be a $\beta$-adic interval $I_\beta$ of order $m+3$ in
$I$ and we have
\begin{equation*}
\lambda(I_\beta) \geq \frac1{\beta^{M+1+m+3}}
  =\frac1{2\beta^{M+4}}\cdot\frac2{\beta^m}
  >\frac{\lambda(I)}{2\beta^{M+4}}.
\end{equation*}
\end{proof}

\subsection{The Algorithm}
\label{sec:beta_algorithm}
\subsubsection*{Notation}
Let $(\beta_j)_{j \geq 1}$ be a sequence of Pisot numbers.
Let $t$ be a positive integer. A \emph{$t$-sequence} is a sequence of intervals $\mathbf{I} = (I_1, \ldots, I_t)$ such that for $1\leq j \leq t$, $I_j$ is $\beta_j$-adic, such that for $1\leq j \leq t-1$, $I_{j+1} \subset I_j$, and such that $\lambda(I_{j+1}) \geq \lambda(I_j) / 2\beta_{j+1}^{M_{\beta_{j+1}}+4}$.
If we have two $\beta$-adic intervals $J \subset I$ then $u_\beta(J)$ means the block of digits that is added to the base $\beta$ expansion of the numbers in $I$ to obtain the $\beta$-expansion of numbers in $J$.
The notation $u_j(\mathbf{J})$ for a $t$-sequence $\mathbf{J}$ shall mean $u_{\beta_j}(J_j)$.
We denoted the dependence on $\beta_j$ of all appearing constants $M$, $n_0$, $C$ and of $\mathcal{L}_n$ explicitly with an $\beta_j$.

\subsubsection*{Input}
Given are values $\epsilon_1=1$, $k_1=1$, $t_1 = 1$ and a sequence $(\beta_j)_{j\geq 1}$ of Pisot numbers $\beta_j$.

\subsubsection*{First step}
Let $\mathbf{I}_1$ be a $t_1$-sequence such that $\mathbf{I}_1 = (I_{1,t_1})$, with $I_{1,t_1} = [0,1)$.
Repeat the bases $\beta_j$ according to conditions
\begin{align}
\max_{1 \leq j \leq t_i} \beta_j & \leq \beta_1 i,\label{cond_on_beta_i_2}\\
\max_{1 \leq j \leq t_i} M_{\beta_j} & \leq (M_{\beta_1} + 1) (1 + \log i), \label{cond_on_beta_i_1} \\
\sum_{1 \leq j \leq t_i} (M_{\beta_j} + 4) \log \beta_j &\leq (M_{\beta_1} + 4)\log \beta_1 (1 + \log i).\label{cond_on_beta_i_last}
\end{align}

\subsubsection*{Step $i+1$ for $i \geq 1$}
From step $i$, we have a $t_i$-sequence $\mathbf{I}_i$ of nested intervals $I_{i,1} \supset \ldots \supset I_{i,t_i}$ where each $I_{i,j}$ is $\beta_j$-adic.

Let
\begin{equation*}
t_{i+1} = \lceil \log(i+1) \rceil, \quad
\epsilon_{i+1} = \frac{1}{t_{i+1}}, \quad
k_{i+1} = t_{i+1},
\end{equation*}
\begin{align*}
\delta_{i+1} &= \frac{1}{2} \frac{1}{2 \beta_1^{M_{\beta_1}+4}} \frac{1}{t_i} \frac{1}{2^{t_i}  \prod_{j\leq t_i} \beta_j^{M_{\beta_j}+4}} \frac{1}{2^{t_{i+1}}  \prod_{j\leq t_{i+1}} \beta_j^{M_{\beta_j}+4}}.
\end{align*}
Choose $n_{i+1}$ to be the least integer such that  
\begin{equation}\label{n_i_admissible_cond}
n_{i+1} \geq \max_{j\leq t_{i+1}} \left( n_{\beta_j}(\epsilon_{i+1}, k_{i+1}) \right),
\end{equation}
and such that for all $1\leq j \leq t_{i+1}$
\begin{equation}\label{n_i_large_enough_cond_for_delta}
\lambda(E^c_n(\epsilon_{i+1}, k_{i+1})) < \delta_{i+1}.
\end{equation}

Furthermore, let
\begin{equation*}
v_i= \left\lceil\max_{j=1,\ldots,t_i}\frac{\log\beta_j}{\log\beta_1}\right\rceil.
\end{equation*}

Then we perform the following steps.
\begin{itemize}
\item Take $L$ to be a $\beta_1$-adic interval of $I_{i,t_i}$ of
length $\lambda(L) \geq \lambda(I_{i,t_i}) 2^{-1}\beta_1^{-(M_{\beta_1}+4)}$. 
\item For each $\beta_1$-adic sub-interval $J_1$ of $L$ with $u_{1}(J_1) = v_{i}n_{i+1}$ find a \\ \noindent
$t_{i+1}$-sequence  $\mathbf{J}=(J_1,\ldots,J_{t_{i+1}})$.
\item Choose the ``leftmost'' of the $t_{i+1}$-sequences $\mathbf{J}$ such
  that $u_{j}(\mathbf{J})$ is $(\varepsilon_{i+1},k_{i+1})$-normal for
  $1\leq j\leq t_i$.
\end{itemize}

\subsubsection*{Output}
The unique real number $X$ in the intersection of all $I_{i,j}$.\\


\noindent We need to show that the algorithm is  well-defined and that the produced number is in fact $\beta_j$-normal for all $j\geq 1$.

\begin{Prop}
This algorithm is well-defined.
\end{Prop}

\begin{proof}
We have to show that in each step $i+1$ there exists at least one $t_{i+1}$-sequence $\mathbf{J}$. Let $\mathcal{S}$ be the union of the intervals $J_{t_{i+1}}$ over the $\vert \mathcal{L}^{\beta_1}_{v_in_{i+1}} \vert$ many $t_{i+1}$-sequences $\mathbf{J}$. By definition of the interval $L$ we have that $\lambda(L) \geq \lambda(I_{i,t_i}) 2^{-1}\beta_1^{-(M_{\beta_1}+4)}$. Furthermore for each sequence we have that $\lambda(J_{t_{i+1}}) \geq2^{-t_{i+1}}\prod_{j=1}^{t_{i+1}}\beta_j^{-(M_{\beta_j}+4)} \lambda(J_1)$. Since the sub-intervals $J_1 \subset L$ form a partition of $L$ we have that $\lambda(\mathcal{S}) \geq2^{-t_{i+1}}\prod_{j=1}^{t_{i+1}}\beta_j^{-(M_{\beta_j}+4)}\lambda(L)$. Combining these inequalities yields 
\begin{equation*}
\lambda(\mathcal{S}) \geq2^{-t_i-t_{i+1}-1}\prod_{j=1}^{t_i}\beta_j^{-(M_{\beta_j}+4)}\prod_{j=1}^{t_{i+1}}\beta_j^{-(M_{\beta_j}+4)} \lambda(I_{i,1}).
\end{equation*}

Now we calculate the measure of the set $\mathcal{N}$ of non-suitable intervals and show that it is less than $\lambda(\mathcal{S})$. For the length of the added word we   have $\Vert u_1(\mathbf{J})\Vert \geq v_in_{i+1}$ and for each $2\leq j\leq t_{i+1}$ we have $\Vert u_j(\mathbf{J}) \Vert \geq n_{i+1}$. 
By the choice of $n_{i+1}$, the subsets of $I_{i,j}$, where   $u_j(\mathbf{J})$ is not $(\varepsilon_{i+1},k_{i+1})$-normal, have Lebesgue measure less that $\delta_{i+1} \lambda(I_{i,j})$, and hence less than $\delta_{i+1} \lambda(I_{i,1})$. Since we consider $t_i$ many bases, we obtain $\lambda(\mathcal{N}) <t_i\delta_{i+1}\lambda(I_{i,1})$.

Combining the estimates of $\mathcal{N}$ and $\mathcal{S}$ we obtain $\lambda(\mathcal{N}) < \lambda(\mathcal{S})$. Since $\mathcal{N}\subset\mathcal{S}$ there must be a $t_{i+1}$-sequence $\mathbf{J}$ such that $u_{j}(\mathbf{J})$ is $(\epsilon_{i+1}, k_{i+1})$-normal for each $1 \leq j \leq t_i$.
\end{proof}

\begin{Thm}\label{thm_beta_normal_algorithm}
Let $(\beta_j)_{j\geq 1}$ be a sequence of Pisot numbers. Then the real number $X$ generated by this algorithm is $\beta_j$-normal for each $j\geq1$.
\end{Thm}

\begin{proof}
We need to verify the growth and normality assumptions of Lemma~\ref{bv:lem4} on the words that correspond to the digits added in each considered base in each step of the algorithm.


To find bounds for the number of added digits in step $i+1$ in base $\beta_j$, for $j\leq t_i$, consider the chain of intervals
\begin{equation*}
I_{i,j} \supset \ldots \supset I_{i,t_i} \supset L \supset J_1 \supset \ldots \supset J_j
\end{equation*}
which is considered in step $i+1$. We find a lower bound on the Lebesgue measure of $J_j$ in the form of
\begin{equation*}
\lambda(J_j) \geq \frac{1}{2^{t_i}} \frac{1}{\beta_1^{M_1+1}} \frac{1}{\beta_1^{v_in_{i+1}}} \prod_{l=1}^{t_i} \frac{1}{\beta_l^{M_{\beta_l} + 4}} \cdot \lambda(I_{i,j}).
\end{equation*}
Thus, Lemma~\ref{liwu:prop2.6} implies for the number $\Vert u_j^{(i+1)}(\mathbf{J}) \Vert$ of digits added in base $\beta_j$, $j\leq t_i$, in step $i+1$ of the algorithm, that
\begin{equation*}
\frac{\log\left(\frac{1}{F_{i+1}} \frac{1}{\beta_j^{M_{\beta_j + 1}}}\right)}{\log \beta_j} \leq \Vert u_j^{(i+1)}(\mathbf{J}) \Vert \leq \frac{\log \frac{1}{F_{i+1}}}{\log \beta_j}
\end{equation*}
where $F_{i+1} = 2^{-t_i} \beta_1^{-(M_1+1)} \beta_1^{-v_i n_{i+1}} \prod_{l=1}^{t_i} \beta_l^{-M_{\beta_l} - 4}$.

Hence $\Vert u_j^{(i+1)}(\mathbf{J}) \Vert \sim \log 1/F_{i+1}$ with implied constants only depending on $\beta_j$. We thus need to show that
\begin{equation*}
\log 1/F_{i+1} = t_i \log 2 + (v_i n_{i+1}+ M_1+1) \log \beta_1  + \sum_{l=1}^{t_i} (M_l + 4) \log \beta_l
\end{equation*}
satisfies assumptions~\eqref{growth_conditions_for_normality} of Lemma~\ref{bv:lem4}. 

We now look at the growth of $n_{i+1}$. In light of Proposition~\ref{non_normal_expl_constants}, condition~\eqref{n_i_admissible_cond} requires 
\begin{equation}
n_{i+1} \geq  M_{\beta_j} + k_{i+1}
\end{equation} 
for all $1\leq j \leq t_{i+1}$. 
We have $\epsilon_{i+1} = 1/t_{i+1} \rightarrow 0$ and $k_{i+1} = t_{i+1} \rightarrow \infty$ as $t_{i+1} \rightarrow \infty$.
Thus also $n_{i+1}$ tends to infinity at least logarithmically in $i$.

Since $\lambda \leq \frac{\beta}{\beta-1} \mu_\beta$, $\beta^k \leq \vert \mathcal{L}_k \vert \leq \frac{\beta}{\beta-1} \beta^k$, and because of Proposition~\ref{non_normal_expl_constants}, condition~\eqref{n_i_large_enough_cond_for_delta} on $n_{i+1}$ is satisfied, if for all $j\leq t_i$,
\begin{equation*}
4 \left(\frac{\beta_j}{\beta_j-1}\right)^2 \beta_j^k \beta_j^{n_{i+1} \eta(\epsilon_{i+1}, k_{i+1})} <\delta_{i+1}.
\end{equation*}
With $\eta$ from equation~\eqref{exponent_eps_k_normal}, this translates into the requirement that for every $j \leq t_i$,
\begin{equation}\label{growth_of_n_i}
n_{i+1} \geq \frac{(M_{\beta_j}+1) \log \beta_j + \log \frac{\beta_j}{\beta_j-1}}{\epsilon_{i+1} \min(\frac{\epsilon_{i+1} \beta_j^{k_{i+1}}}{16}, \frac34)} \left( \log\left(4 \left(\frac{\beta_j}{\beta_j - 1} \right)^2 \beta_j^{k_{i+1}}\right) + \log \frac{1}{\delta_{i+1}} \right),
\end{equation}
where
\begin{align*}
\log \frac{1}{\delta_{i+1}} = & 2 \log 2 + \log t_i + (t_i + t_{i+1}) \log 2 + (M_{\beta_1} + 4) \log \beta_1 \\
&+ 2 \sum_{1 \leq j \leq t_i} (M_{\beta_j} + 4) \log \beta_j + \sum_{t_i < j \leq t_{i+1}} (M_{\beta_j} + 4) \log \beta_j
\end{align*}
(where the last sum is empty if $t_i = t_{i+1}$).


Properties~\eqref{cond_on_beta_i_2} and~\eqref{cond_on_beta_i_last} on the sequence $(\beta_j)_{j\geq1}$ imply
\begin{equation}\label{cond_on_beta_derived}
\max_{1\leq j \leq t_i} \left((M_{\beta_j}+1)\log \beta_j + \log \frac{\beta_j}{\beta_j - 1} \right) 
\leq \left((M_{\beta_1}+4)\log \beta_1 + \log \frac{\beta_1}{\sqrt[3]{2}-1} + 1 \right)(1+\log i).
\end{equation}

Conditions~\eqref{cond_on_beta_i_2}~-~\eqref{cond_on_beta_i_last} can be achieved by suitably repeating the bases $\beta_j$. All conditions are satisfied in step $1$, and the process of repeating the bases is possible computably.

Properties~\eqref{cond_on_beta_i_2}~-~\eqref{cond_on_beta_i_last} and \eqref{cond_on_beta_derived}, together with $t_{i+1} = k_{i+1} = 1/\epsilon_{i+1} \sim \log i$, imply that for $i$ large enough
\begin{equation*}
n_{i+1} \geq O\left(\frac{\log i}{1/\log i} \left(\log i + (\log i)^2 + \log \log i + \log i + \log i \right) \right) = O\left( (\log i)^4 \right)
\end{equation*}
where the implied constant only depends on $\beta_1$. Hence $n_{i+1}$ grows at least as $O(\log i)$ and at most as $O((\log i)^4)$, where the implied constants depend only on $\beta_1$. Thus $\log 1/F_{i+1}$ and hence also $\Vert u_j^{(i+1)}(\mathbf{J})\Vert$ growths at least as $O(\log i)$ and at most as $O((\log i)^4)$, where again the implied constants only depend on $\beta_1$. Thus $\Vert u_j^{(i+1)}(\mathbf{J})\Vert$ satisfies conditions~\eqref{growth_conditions_for_normality} of Proposition~\ref{bv:lem4}. Hence the number $X$ produced by 
this algorithm
is $\beta_j$-normal for every $j\geq1$.
\end{proof}

\subsection*{Remark} The choices of how $t_i$, $\epsilon_i$ and $k_i$ change with the step $i$ of the algorithm and the conditions on the sequence of bases $(\beta_j)_{j \geq 1}$ are rather arbitrary. There is a lot of freedom to optimize for other quantities, such as done in Becher, Heiber, Slaman \cite{becher_heiber_slaman2013:polynomial_time_algorithm} where computational speed is optimized. This is not taken into account here.

\subsection*{Remark} 

%

Following these lines, an extension of Becher, Heiber, Slaman's algorithm to a countable set of real bases that are $\beta$-numbers is possible, provided these bases are bounded away from $1$ and such there is a uniform bound on the length of the periodic part in their orbit of $1$. 

A $\beta$-number is a real number $\beta$ such that the orbit of $1$ under $T_\beta$ is finite. Pisot numbers are $\beta$-numbers. It is not known under which conditions Salem numbers are or are not $\beta$-numbers (a \emph{Salem number} is a real algebraic integer $\beta > 1$ such that all its conjugates have absolute values at most equal to one, with equality in at least one case). Salem numbers of degree $4$ are $\beta$-numbers, but there is computational and heuristic evidence that higher degree Salem numbers exist that are no $\beta$-numbers, see for example \cite{boyd1995:salem}.

Note that $\beta$-numbers satisfy the specification property - one can always use a block of zeros to make the concatenation of two admissible blocks admissible. This is because admissible words can be characterized as precisely the subwords of the lexicographic largest word in the $\beta$-shift. Since the orbit of $1$ is finite, this word will be eventually periodic and hence the lengths of subwords consisting of only zeros is bounded. Thus Lemma~3 in \cite{bertrand-mathis_volkmann1989:epsilon_k_normal} on the number of $(\epsilon,k)$-normal admissible words is valid and can be used as an existence criterion for a $t_i$ sequence $\mathbf{J}$ in each step of the algorithm.

Note also that $\beta$-numbers also satisfy Proposition 2.6 of \cite{li_wu2008:beta_expansion_and} needed to control the decay of the length of subintervals. However, we are looking for a lower bound for the measure of cylinder intervals of the form~\eqref{lower_bound_measure_beta_adic_interval} that is uniform for all bases $\beta$ under consideration. This can achieved by requiring that there is a uniform bound on the length of the period of the orbit of $1$ under $T_\beta$ for each $\beta$ under consideration.

When adapting the proof of Lemma~\ref{inscribed_interval_beta} to $\beta$-numbers, we moreover need to require that the set of $\beta$-numbers under consideration is bounded away from $1$, as above with the plastic number.

\section{Explicit Estimates for $\beta$-expansions}
\label{sec:ExplicitEstimates}


\noindent In this section we make explicit the constants in Lemma \ref{bv:lem3} using large deviation estimates for certain dependent random variables. This requires us to provide an upper bound for the length of the largest block of zeros appearing in the modified $\beta$-expansion of $1$ for a Pisot number $\beta$. 

\subsection{Number of zeros in the expansion of $1$}\label{subsec:number_of_zeros}

Let $\beta$ be a Pisot number and denote by $d_\beta(1) = 0.\epsilon_1 \epsilon_2 \ldots$ the $\beta$-expansion of $1$, i.e. $\epsilon_1 = \lfloor \beta \rfloor$ and $\epsilon_i = \lfloor \beta T_\beta^{i-1}(1) \rfloor$ for $i\geq 1$. Let $d^\ast_\beta(1)$ be the modified $\beta$-expansion of $1$, i.e. $d^\ast_\beta(1) = d_\beta(1)$ if the sequence $\epsilon_1 \epsilon_2 \ldots$ does not end with infinitely many zeros, and $d^\ast_\beta(1) = 0.(\epsilon_1 \epsilon_2 \ldots \epsilon_{n-1} (\epsilon_n-1))^\omega$ when $d_\beta(1)$ ends in infinitely many zeros and $\epsilon_n$ is the last non-zero digit. It is known that $d^\ast_\beta(1)$ is purely periodic or eventually periodic if $\beta$ is Pisot. We reprove this fact here and give an explicit upper bound for the preperiod length $v$ and period length $p$ and take $v+p$ as a trivial upper bound for the size of the largest block of zeros in $d^\ast_\beta(1)$. Note that $d^\ast_\beta(1)$ is (eventually) periodic if the orbit of $1$ under $T_\beta$ is finite, and that the number of distinct elements in this orbit is precisely $v+p$.

\begin{Prop}\label{Proposition_on_M}
Let $\beta$ be a Pisot number of degree $d$ with $r$ real conjugates $\beta = \beta_1, \beta_2, \ldots, \beta_r$ and $2s$ complex conjugates $\beta_{r+1}, \ldots, \beta_d$. Then the orbit of $1$ under the map $T_\beta$, i.e. the set
\begin{equation*}
\{T_\beta^k(1) \mid k\geq 0 \},
\end{equation*}
is finite and its number of elements is bounded by 
\begin{equation}\label{M}
M = d! \det(B)^{-1} 2^{r+s-1} \pi^s C^{r+2s-1} + d
\end{equation}
where
\begin{equation}\label{B}
B = \begin{pmatrix}
1 & \beta & \ldots & \beta^{d-1} \\
1 & \beta_2 & \ldots & \beta_2^{d-1} \\
\vdots & \vdots & & \vdots\\
1 & \beta_d & \dots & \beta_d^{d-1}
\end{pmatrix}
\end{equation}
and where
\begin{equation*}
C = 1 + \frac{\lfloor \beta \rfloor}{1-\eta}
\end{equation*}
with $\eta = \max_{2\leq j \leq d} \vert \beta_j \vert < 1$.
\end{Prop}

\begin{proof}
For $k\geq 0$, $T^k_\beta(1)$ is an element of $\mathbb{Z}[\beta]$, hence there is a unique representation $T^k_\beta(1) = p_0^{(k)} + p_1^{(k)} \beta + \ldots + p_{d-1}^{(k)} \beta^{d-1}$ with $p_i^{(k)} \in \mathbb{Z}$. Denote by $\sigma_j$, $1\leq j \leq d$, the $j$-th conjugation, ordered such that the first $r$ are real, and $\sigma_{r+i} = \bar{\sigma}_{r+s+i}$ for $1\leq i \leq s$. We have 
\begin{equation*}
T^k_\beta(1) = \beta^k \left( 1 - \sum_{l=1}^k \epsilon_l \beta^{-l} \right)
\end{equation*}
hence for $2\leq j \leq d$
\begin{align*}
\vert \sigma_j(T_\beta^k(1))\vert \leq 1 + \frac{\lfloor \beta \rfloor}{1- \eta}
\end{align*}
where $\eta = \max_{2\leq j \leq d} \vert \beta_j \vert < 1$.

Note that 
\begin{equation*}
B \begin{pmatrix}
p_0^{(k)} \\ p_1^{(k)} \\ \vdots \\ p_{d-1}^{(k)}
\end{pmatrix} 
= \begin{pmatrix}
T_\beta^{k}(1) \\ \sigma_2(T_\beta^k(1)) \\ \vdots \\ \sigma_d(T_\beta^k(1))
\end{pmatrix}
\end{equation*}
where $B$ is as in~\eqref{B} and has determinant $\text{det} B = \prod_{1\leq i < j \leq d} (\beta_j - \beta_i) \neq 0$.
Now, since the vector of $T_\beta^k(1)$ and its conjugates can be canonically embedded in a compact convex set in $\mathbb{R}^{r+2s}$ of volume $2^{r+s-1} \pi^s C^{r+2s-1}$, we can count the $\mathbb{Z}^d$-lattice points in a compact convex set in $\mathbb{R}^{d}$ of volume $\det(B)^{-1} 2^{r+s-1} \pi^s C^{r+2s-1}$. By loosing a factor of $2$, we can make this set additionally centrally symmetric if we allow $T_\beta^k(1)$ (formally) to take on values in the interval $[-1,1]$. Then we can use a result by Blichfeldt \cite{blichfeldt} and bound the number of $\mathbb{Z}^d$-lattice points in $B^{-1} Y$ by
\begin{equation*}
\vert B^{-1} Y \cap \mathbb{Z}^d \vert \leq d! \det(B)^{-1} 2^{r+s-1} \pi^s C^{r+2s-1} + d
\end{equation*}
with $C = 1 + \frac{\lfloor \beta \rfloor}{1-\eta}$ and hence obtain an upper bound for the number of distinct points in the orbit of $1$ under $T_\beta$ which is also a trivial upper bound for the maximum number of consecutive zeros in the modified $\beta$-expansion of $1$ as explained above.
\end{proof}


\subsection{Number of not $(\epsilon,k)$-normal numbers}\label{large_deviation_section}

Let $\beta$ be a Pisot number and let $\mathcal{L}_n$ be the set of all admissible words of length $n$. Fix $\epsilon>0$ and a positive integer $k$. We wish to find explicit estimates for the number of non-$(\epsilon,k)$-normal words of length $n$ for fixed $\epsilon>0$ and $k$ such as given in Lemma \ref{bv:lem3} (Lemma 3 in \cite{bertrand-mathis_volkmann1989:epsilon_k_normal}). The method in \cite{bertrand-mathis_volkmann1989:epsilon_k_normal} uses methods of ergodic theory and the authors are not aware of a method to make the implied constants explicit. Therefore we use a probabilistic approach by viewing the digits to base $\beta$ as random variables and using a variant of Hoeffding's inequality for dependent random variables to bound the tail distribution of their sum. This approach automatically gives all involved constants explicitly. We use the following Lemma due to Siegel (Theorem 5 in \cite{siegel:towards_a_usable_theory_of_chernoff}).

\begin{Lem}\label{Hoeffdingvariant}
Let $X = X_1 + X_2 	+ \ldots + X_l$ be the sum of $l$ possibly dependent random variables. Suppose that $X_i$, for $i = 1,2, \ldots, l$, is the sum of $n_i$ mutually independent random variables having values in the interval $[0,1]$. Let $\mathbb{E}[X_i] = n_i p_i$. Then for $a\geq 0$
\begin{equation*}
\mathbb{P}(X-\mathbb{E}[X] \geq a) < \exp\left(- \frac{a^2}{8(\sum_i \sqrt{p_i(1-p_i)n_i})^2} \right)
+ \exp \left( - \frac{3a}{4 \sum_i(1-p_i)^2} \right).
\end{equation*}
\end{Lem}

\begin{Prop}\label{non_normal_expl_constants}
Let $\beta$ be a Pisot number. The $\mu_\beta$-measure of the set of not $(\epsilon,k)$-normal words of length $n$ satisfies
\begin{equation*}
 \mu_\beta(E^c_n(\epsilon,k)) \leq 4 \vert \mathcal{L}_k \vert \vert \mathcal{L}_n \vert^{-\eta}
\end{equation*}
for $n \geq M + k$ with $\eta>0$ as in equation~\eqref{exponent_eps_k_normal} and $M$ as in equation~\eqref{M}.
\end{Prop}

\begin{proof}
Let $d \in \mathcal{L}_k$ and for $n \geq M+k$, let $X_1, \ldots, X_{M+1} : \mathcal{L}_n \rightarrow \mathbb{R}$ be random variables where $X_i(\omega)$ denotes the number of occurrences of the word $d$ in $\omega = \omega_1 \ldots \omega_n$ at positions 
\begin{equation*}
\omega_{i+j(M+1)} \omega_{i+j(M+1)+1} \ldots \omega_{i+j(M+1)+k-1}
\end{equation*}
for $0\leq j \leq  \lfloor \frac{n-k}{M+1} \rfloor$. The $X_i$ are dependent, but each is a sum of $n_i = \lfloor \frac{n-k}{M+1} \rfloor + 1$ independent identically distributed random variables $Y_j^{(i)}$ that take value one if and only if the word $d$ appears in $\omega$ starting at digit $\omega_{i+j(M+1)}$ and zero otherwise. We have $\mathbb{E}[X] = n \mu_\beta(c(d))$ and $\mathbb{E}[X_i] = n_i \mu_\beta(c(d))$. 
Denote by $\bar{E}_n(\epsilon,k)$ the set of words of length $n$ for which there is a subword $d$ of length $k$ that appears more often than $n(\mu_\beta(c(d))+\epsilon)$ times and let $\bar{E}_n(\epsilon,d)$ be the set of words of length $n$ for which the subword $d$ appears more often than $n(\mu_\beta(c(d))+\epsilon)$ times.
We apply Lemma \ref{Hoeffdingvariant} with $l=M+1$, $n_i$ as above, $p_i = \mu_\beta(c(d))$ and $a = n\epsilon$ and obtain
\begin{align*}
\mu_\beta(X > n(\mu_\beta(c(d)) + \epsilon)) &= \mu_\beta(\bar{E}_n(\epsilon,d)) \\
 &< \exp \left( - \frac{(n\epsilon)^2}{8 \mu_\beta(c(d))(1-\mu_\beta(c(d))) (M+1)^2 (\lfloor \frac{n-k}{M+1} \rfloor + 1) } \right)\\
&\quad + \exp \left(- \frac{3 n \epsilon}{4(M+1)(1-\mu_\beta(c(d)))^2} \right).
\end{align*}
Using $c \beta^{-k} \leq \mu_\beta(c(d)) \leq \beta^{-k}$ and $n\geq M+1$, this is
\begin{equation*}
< \exp\left( - \frac{ \epsilon^2 n}{16(M+1) \beta^{-k}} \right)
+ \exp\left( - \frac{3\epsilon n}{4(M+1)} \right)
< 2 \exp\left(-\frac{\epsilon n}{M+1} \min(\frac{\epsilon}{16 \beta^{-k}}, \frac{3}{4}) \right).
\end{equation*}
Finally, since $\mu_\beta( \bar{E}_n(\epsilon,k)) \leq \sum_{d \in \mathcal{L}_k} \mu_\beta (\bar{E}_n(\epsilon,d) )$ and using that $\beta^n \leq \vert \mathcal{L}_n \vert \leq \frac{\beta}{\beta - 1} \beta^n$ we obtain
\begin{align*}
\mu_\beta ( \bar{E}_n(\epsilon,k) ) 
&\leq  \vert \mathcal{L}_k \vert 2 \exp\left(-\frac{\epsilon n}{M+1} \min(\frac{\epsilon}{16 \beta^{-k}}, \frac{3}{4}) \right)\\
& \leq 2 \vert \mathcal{L}_k \vert \vert \mathcal{L}_n \vert^{-\eta}
\end{align*}
with
\begin{equation}\label{exponent_eps_k_normal}
\eta = \frac{\epsilon \min(\frac{\epsilon}{16 \beta^{-k}}, \frac{3}{4})}{\log(\frac{\beta}{\beta - 1}) + (M+1) \log \beta} > 0.
\end{equation}

Using the same argument with $Y=n-X$ gives a symmetrical upper bound for the number of words $\omega$ of length $n$ in which the word $d$ appears less than $n \mu_\beta(c(d)) - \epsilon n$ times. Thus we obtain an upper bound for the number of not $(\epsilon,k)$-normal words of length $n$ of the form
\begin{equation*}
4 \vert \mathcal{L}_k \vert \vert \mathcal{L}_n \vert^{-\eta}
\end{equation*}
for $n \geq M + k$ with $\eta$ as in~\eqref{exponent_eps_k_normal}.
\end{proof}

\begin{Cor}
The number of not $(\epsilon,k)$-normal words of length $n$ satisfies
\begin{equation*}
\vert E^c_n(\epsilon,k) \vert \leq C \vert \mathcal{L}_n \vert^{1-\eta}
\end{equation*}
for $n \geq M + k$ with $\eta>0$ as in equation~\eqref{exponent_eps_k_normal}, $M$ as in equation~\eqref{M}, and where $C = 4 \vert \mathcal{L}_k \vert \beta^{M+1}\frac{\beta}{\beta-1}$.
\end{Cor}

\begin{proof}
Since the Parry measure $\mu_\beta$ satisfies
\begin{equation*}
\left(1 - \frac{1}{\beta}\right) \lambda \leq \mu_\beta \leq \frac{\beta}{\beta-1} \lambda
\end{equation*}
with respect to the Lebesgue measure $\lambda$, and due to the bounds on the Lebesgue measure of $\beta$-adic cylinder intervals from Lemma~\ref{liwu:prop2.6}, the bound from Proposition~\ref{non_normal_expl_constants} on the $\mu_\beta$ measure of the set of non-$(\epsilon,k)$-normal words of length $n$ implies for the number of such words
\begin{equation}\label{number_of_non_normal_words}
\vert E^c_n(\epsilon,k) \vert \leq C \vert \mathcal{L}_n \vert^{1-\eta},
\end{equation}
where $C = C(\beta,k) = 4 \vert \mathcal{L}_k \vert \beta^M \frac{\beta}{\beta-1}$ and $\eta = \eta(\beta, \epsilon, k)$ as given in equation~\eqref{exponent_eps_k_normal} and where we used that $\beta^n \leq \vert \mathcal{L}_n \vert \leq \frac{\beta}{\beta-1} \beta^n$.
\end{proof}

\subsubsection*{Acknowledgements} For the realization of the present paper the first author received support from the Conseil Régional de Lorraine. Parts of this research work were done when the first author was visiting the Department of Mathematics of Graz University of Technology. The author thanks the institution for their hospitality.
The second author was supported by the Austrian Science Fund (FWF): I 1751-N26; W1230, Doctoral Program ``Discrete Mathematics''; and  SFB F 5510-N26. He would like to thank Karma Dajani and Bing Li for some interesting discussions on $\beta$-expansions.


%
%
%
%
%
%
%
%
%
%
%



\begin{bibdiv}
\begin{biblist}

\bib{alvarez_becher2015:levin}{article}{
	AUTHOR = {Alvarez, Nicol{\'a}s},	
	AUTHOR = {Becher, Ver{\'o}nica},
	TITLE = {M. Levin's construction of absolutely normal numbers with very low discrepancy},
	JOURNAL = {arXiv:1510.02004},
	URL = {http://arxiv.org/abs/1510.02004},
}

\bib{bailey2014:normality_of_constants}{article}{
    AUTHOR = {Bailey, David H.},
    AUTHOR = {Borwein, Jonathan},
     TITLE = {Pi {D}ay is upon us again and we still do not know if pi is
              normal},
   JOURNAL = {Amer. Math. Monthly},
  FJOURNAL = {American Mathematical Monthly},
    VOLUME = {121},
      YEAR = {2014},
    NUMBER = {3},
     PAGES = {191--206},
      ISSN = {0002-9890},
   MRCLASS = {11K16 (11A63)},
  MRNUMBER = {3168990},
MRREVIEWER = {Manfred G. Madritsch},
       DOI = {10.4169/amer.math.monthly.121.03.191},
       URL = {http://dx.doi.org/10.4169/amer.math.monthly.121.03.191},
}
	

\bib{becher_bugeaud_slaman2013:simply_normal}{article}{
	AUTHOR = {Becher, Ver{\'o}nica},
	AUTHOR = {Bugeaud, Yann},
	AUTHOR = {Slaman, Theodore A.},
	TITLE  = {On Simply Normal Numbers to Different Bases},
	JOURNAL = {arXiv:1311.0332},
	YEAR = {2013}
	URL = {http://arxiv.org/abs/1311.0332},
}

\bib{becher_figueira:2002}{article}{
	AUTHOR = {Becher, Ver{\'o}nica},
    AUTHOR = {Figueira, Santiago},
    TITLE  = {An example of a computable absolutely normal number},
    JOURNAL = {Theoretical Computer Science},
    VOLUME = {270},
    YEAR = {2002},
    PAGES = {126--138},
}

\bib{becher_figueira_picchi2007:turing_unpublished}{article}{
	AUTHOR = {Becher, Ver{\'o}nica},
    AUTHOR = {Figueira, Santiago},
    AUTHOR = {Picchi, Rafael},
     TITLE = {Turing's unpublished algorithm for normal numbers},
   JOURNAL = {Theoret. Comput. Sci.},
  FJOURNAL = {Theoretical Computer Science},
    VOLUME = {377},
      YEAR = {2007},
    NUMBER = {1-3},
     PAGES = {126--138},
      ISSN = {0304-3975},
     CODEN = {TCSDI},
   MRCLASS = {03D80 (01A60 03-03 11K16 11Y16 68Q30)},
  MRNUMBER = {2323391 (2008j:03064)},
MRREVIEWER = {George Barmpalias},
       DOI = {10.1016/j.tcs.2007.02.022},
       URL = {http://dx.doi.org/10.1016/j.tcs.2007.02.022},
}

\bib{becher_heiber_slaman2013:polynomial_time_algorithm}{article}{
      author={Becher, Ver{\'o}nica},
      author={Heiber, Pablo~Ariel},
      author={Slaman, Theodore~A.},
       title={A polynomial-time algorithm for computing absolutely normal
  numbers},
        date={2013},
        ISSN={0890-5401},
     journal={Inform. and Comput.},
      volume={232},
       pages={1\ndash 9},
         url={http://dx.doi.org/10.1016/j.ic.2013.08.013},
      review={\MR{3132518}},
}

\bib{Becher_slaman2014:on_the_normality_of_numbers_to_different_bases}{article}{
    AUTHOR = {Becher, Ver{\'o}nica},
    AUTHOR = {Slaman, Theodore A.},
     TITLE = {On the normality of numbers to different bases},
   JOURNAL = {J. Lond. Math. Soc. (2)},
  FJOURNAL = {Journal of the London Mathematical Society. Second Series},
    VOLUME = {90},
      YEAR = {2014},
    NUMBER = {2},
     PAGES = {472--494},
      ISSN = {0024-6107},
   MRCLASS = {11K16 (03E15)},
  MRNUMBER = {3263961},
MRREVIEWER = {Christoph Baxa},
       DOI = {10.1112/jlms/jdu035},
       URL = {http://dx.doi.org/10.1112/jlms/jdu035},
}

\bib{bertrand-mathis_volkmann1989:epsilon_k_normal}{article}{
      author={Bertrand-Mathis, Anne},
      author={Volkmann, Bodo},
       title={On {$(\epsilon,k)$}-normal words in connecting dynamical
  systems},
        date={1989},
        ISSN={0026-9255},
     journal={Monatsh. Math.},
      volume={107},
      number={4},
       pages={267\ndash 279},
         url={http://dx.doi.org/10.1007/BF01517354},
      review={\MR{1012458 (90m:11115)}},
}

\bib{besicovitch1935:epsilon}{article}{
author={Besicovitch, A. S.},
title={The asymptotic distribution of the numerals in the decimal representation of the squares of the natural numbers},
year={1935},
pages={146–156},
journal={Math. Zeit.},
number={39},
}


\bib{blichfeldt}{article}{
author={Blichfeldt, H. F.},
title={Notes on Geometry of Numbers},
year={1921},
pages={150–153},
journal={Bull. Amer. Math. Soc.},
number={4},
volume={27},
series={The october meeting of the San Francisco section},
}

\bib{borel1909}{article}{
year={1909},
issn={0009-725X},
journal={Rendiconti del Circolo Matematico di Palermo},
volume={27},
number={1},
doi={10.1007/BF03019651},
title={Les probabilit\'{e}s d\'{e}nombrables et leurs applications arithm\'{e}tiques},
url={http://dx.doi.org/10.1007/BF03019651},
publisher={Springer-Verlag},
author={Borel, \'{E}mile},
pages={247-271},
language={French}
}

\bib{boyd1995:salem}{article}{
    AUTHOR = {Boyd, David W.},
     TITLE = {On the beta expansion for {S}alem numbers of degree {$6$}},
   JOURNAL = {Math. Comp.},
  FJOURNAL = {Mathematics of Computation},
    VOLUME = {65},
      YEAR = {1996},
    NUMBER = {214},
     PAGES = {861--875, $S$29--$S$31},
      ISSN = {0025-5718},
     CODEN = {MCMPAF},
   MRCLASS = {11K16 (11R06)},
  MRNUMBER = {1333306 (96g:11091)},
MRREVIEWER = {Christopher Smyth},
       DOI = {10.1090/S0025-5718-96-00700-4},
       URL = {http://dx.doi.org/10.1090/S0025-5718-96-00700-4},
}

\bib{bugeaud2012distribution}{book}{
  title={Distribution Modulo One and Diophantine Approximation},
  author={Bugeaud, Y.},
  isbn={9780521111690},
  lccn={2012013417},
  series={Cambridge Tracts in Mathematics},
  year={2012},
  publisher={Cambridge University Press},
}

\bib{chaitin}{article}{
    AUTHOR = {Chaitin, Gregory J.},
     TITLE = {A theory of program size formally identical to information
              theory},
   JOURNAL = {J. Assoc. Comput. Mach.},
  FJOURNAL = {Journal of the Association for Computing Machinery},
    VOLUME = {22},
      YEAR = {1975},
     PAGES = {329--340},
      ISSN = {0004-5411},
   MRCLASS = {94A15 (68A20)},
  MRNUMBER = {0411829},
MRREVIEWER = {Aldo De Luca},
}

\bib{champernowne}{article}{
    AUTHOR = {Champernowne, D. G.},
     TITLE = {The {C}onstruction of {D}ecimals {N}ormal in the {S}cale of {T}en},
   JOURNAL = {J. London Math. Soc.},
  FJOURNAL = {The Journal of the London Mathematical Society},
    VOLUME = {S1-8},
    NUMBER = {4},
     PAGES = {254},
   MRCLASS = {Contributed Item},
  MRNUMBER = {1573965},
       DOI = {10.1112/jlms/s1-8.4.254},
       URL = {http://dx.doi.org/10.1112/jlms/s1-8.4.254},
}

\bib{davenport_erdos:note_on_normal}{article}{
AUTHOR = {Davenport, H.},
AUTHOR = {Erd{\"o}s, P.},
     TITLE = {Note on normal decimals},
   JOURNAL = {Canadian J. Math.},
  FJOURNAL = {Canadian Journal of Mathematics. Journal Canadien de
              Math\'ematiques},
    VOLUME = {4},
      YEAR = {1952},
     PAGES = {58--63},
      ISSN = {0008-414X},
   MRCLASS = {10.0X},
  MRNUMBER = {0047084 (13,825g)},
MRREVIEWER = {J. F. Koksma},
}

\bib{dufresnoy_pisot1955:etude_de_certaines}{article}{
      author={Dufresnoy, J.},
      author={Pisot, Ch.},
       title={Etude de certaines fonctions m\'eromorphes born\'ees sur le
  cercle unit\'e. {A}pplication \`a un ensemble ferm\'e d'entiers
  alg\'ebriques},
        date={1955},
        ISSN={0012-9593},
     journal={Ann. Sci. Ecole Norm. Sup. (3)},
      volume={72},
       pages={69\ndash 92},
      review={\MR{0072902 (17,349d)}},
}

\bib{fukuyama2013:metric_discrepancy}{article}{
    AUTHOR = {Fukuyama, Katusi},
     TITLE = {Metric discrepancy results for alternating geometric
              progressions},
   JOURNAL = {Monatsh. Math.},
  FJOURNAL = {Monatshefte f\"ur Mathematik},
    VOLUME = {171},
      YEAR = {2013},
    NUMBER = {1},
     PAGES = {33--63},
      ISSN = {0026-9255},
   MRCLASS = {11K38 (42A55 60F15)},
  MRNUMBER = {3066814},
MRREVIEWER = {Wolfgang Steiner},
       DOI = {10.1007/s00605-012-0419-4},
       URL = {http://dx.doi.org/10.1007/s00605-012-0419-4},
}

\bib{gaalgal1964:discrepancy}{article}{
    AUTHOR = {Gaal, S.},
    AUTHOR = {G{\'a}l, L.},
     TITLE = {The discrepancy of the sequence {$\{(2^{n}x)\}$}},
   JOURNAL = {Nederl. Akad. Wetensch. Proc. Ser. A 67 = Indag. Math.},
    VOLUME = {26},
      YEAR = {1964},
     PAGES = {129--143},
   MRCLASS = {40.10 (10.33)},
  MRNUMBER = {0163089 (29 \#392)},
MRREVIEWER = {E. J. Akutowicz},
}


\bib{levin1979:absolutely_normal}{article}{
    AUTHOR = {Levin, M. B.},
     TITLE = {Absolutely normal numbers},
   JOURNAL = {Vestnik Moskov. Univ. Ser. I Mat. Mekh.},
  FJOURNAL = {Vestnik Moskovskogo Universiteta. Seriya I. Matematika,
              Mekhanika},
      YEAR = {1979},
    NUMBER = {1},
     PAGES = {31--37, 87},
      ISSN = {0201-7385},
   MRCLASS = {10K25},
  MRNUMBER = {525299 (80d:10076)},
MRREVIEWER = {J. Galambos},
}

\bib{levin1999:discrepancy}{article}{
    AUTHOR = {Levin, M. B.},
     TITLE = {On the discrepancy estimate of normal numbers},
   JOURNAL = {Acta Arith.},
  FJOURNAL = {Acta Arithmetica},
    VOLUME = {88},
      YEAR = {1999},
    NUMBER = {2},
     PAGES = {99--111},
      ISSN = {0065-1036},
     CODEN = {AARIA9},
   MRCLASS = {11K16 (11K38)},
  MRNUMBER = {1700240 (2000j:11115)},
MRREVIEWER = {Henri Faure},
}

\bib{li_wu2008:beta_expansion_and}{article}{
      author={Li, Bing},
      author={Wu, Jun},
       title={Beta-expansion and continued fraction expansion},
        date={2008},
        ISSN={0022-247X},
     journal={J. Math. Anal. Appl.},
      volume={339},
      number={2},
       pages={1322\ndash 1331},
         url={http://dx.doi.org/10.1016/j.jmaa.2007.07.070},
      review={\MR{2377089 (2008m:11148)}},
}

\bib{madritsch_thuswaldner_tichy}{article}{
    AUTHOR = {Madritsch, Manfred G.},
	AUTHOR = {Thuswaldner, J{\"o}rg M.},
	AUTHOR = {Tichy, Robert F.},
     TITLE = {Normality of numbers generated by the values of entire
              functions},
   JOURNAL = {J. Number Theory},
  FJOURNAL = {Journal of Number Theory},
    VOLUME = {128},
      YEAR = {2008},
    NUMBER = {5},
     PAGES = {1127--1145},
      ISSN = {0022-314X},
     CODEN = {JNUTA9},
   MRCLASS = {11K16},
  MRNUMBER = {2406483 (2009f:11092)},
MRREVIEWER = {R. C. Baker},
       DOI = {10.1016/j.jnt.2007.04.005},
       URL = {http://dx.doi.org/10.1016/j.jnt.2007.04.005},
}

\bib{nakai_shiokawa:class_of_normal_numbers}{article}{
    AUTHOR = {Nakai, Yoshinobu},
    AUTHOR = {Shiokawa, Iekata},
     TITLE = {A class of normal numbers},
   JOURNAL = {Japan. J. Math. (N.S.)},
  FJOURNAL = {Japanese Journal of Mathematics. New Series},
    VOLUME = {16},
      YEAR = {1990},
    NUMBER = {1},
     PAGES = {17--29},
      ISSN = {0289-2316},
     CODEN = {JJMAAK},
   MRCLASS = {11K16},
  MRNUMBER = {1064444 (91g:11081)},
MRREVIEWER = {W. W. Adams},
}




\bib{parry1960:beta_expansions}{article}{
    AUTHOR = {Parry, W.},
     TITLE = {On the {$\beta $}-expansions of real numbers},
   JOURNAL = {Acta Math. Acad. Sci. Hungar.},
  FJOURNAL = {Acta Mathematica Academiae Scientiarum Hungaricae},
    VOLUME = {11},
      YEAR = {1960},
     PAGES = {401--416},
      ISSN = {0001-5954},
   MRCLASS = {28.70 (10.09)},
  MRNUMBER = {0142719 (26 \#288)},
MRREVIEWER = {A. R{\'e}nyi},
}

\bib{philipp1975:lacunary_limit}{article}{
    AUTHOR = {Philipp, Walter},
     TITLE = {Limit theorems for lacunary series and uniform distribution
              {${\rm mod}\ 1$}},
   JOURNAL = {Acta Arith.},
  FJOURNAL = {Polska Akademia Nauk. Instytut Matematyczny. Acta Arithmetica},
    VOLUME = {26},
      YEAR = {1974/75},
    NUMBER = {3},
     PAGES = {241--251},
      ISSN = {0065-1036},
   MRCLASS = {10K30 (10K05)},
  MRNUMBER = {0379420 (52 \#325)},
MRREVIEWER = {O. P. Stackelberg},
}

\bib{renyi1957:reps_for_real_numbers}{article}{
    AUTHOR = {R{\'e}nyi, A.},
     TITLE = {Representations for real numbers and their ergodic properties},
   JOURNAL = {Acta Math. Acad. Sci. Hungar},
  FJOURNAL = {Acta Mathematica Academiae Scientiarum Hungaricae},
    VOLUME = {8},
      YEAR = {1957},
     PAGES = {477--493},
      ISSN = {0001-5954},
   MRCLASS = {10.00},
  MRNUMBER = {0097374 (20 \#3843)},
MRREVIEWER = {W. J. Thron},
}

\bib{scheerer2015:schmidt}{article}{
author = {{Scheerer}, A.-M.},
    title = {Computable Absolutely Normal Numbers and Discrepancies},
     year = {2015},   
     JOURNAL = {arXiv:1511.03582},
	URL = {http://arxiv.org/abs/1511.03582},
}

\bib{scheerer2015:pisot}{article}{
author = {{Scheerer}, A.-M.},
    title = {Normality in Pisot Numeration Systems},
     year = {2015},   
     JOURNAL = {arXiv:1503.08047},
	URL = {http://arxiv.org/abs/1503.08047},
}

\bib{schiffer:discrepancy}{article}{
    AUTHOR = {Schiffer, Johann},
     TITLE = {Discrepancy of normal numbers},
   JOURNAL = {Acta Arith.},
  FJOURNAL = {Polska Akademia Nauk. Instytut Matematyczny. Acta Arithmetica},
    VOLUME = {47},
      YEAR = {1986},
    NUMBER = {2},
     PAGES = {175--186},
      ISSN = {0065-1036},
     CODEN = {AARIA9},
   MRCLASS = {11K16},
  MRNUMBER = {867496 (88d:11072)},
MRREVIEWER = {R. G. Stoneham},
}

\bib{schmidt1960:on_normal_numbers}{article}{
	AUTHOR = {Schmidt, Wolfgang M.},
     TITLE = {On normal numbers},
   JOURNAL = {Pacific J. Math.},
  FJOURNAL = {Pacific Journal of Mathematics},
    VOLUME = {10},
      YEAR = {1960},
     PAGES = {661--672},
      ISSN = {0030-8730},
   MRCLASS = {10.00},
  MRNUMBER = {0117212 (22 \#7994)},
MRREVIEWER = {F. Herzog},
}

\bib{schmidt1961:uber_die_normalitat}{article}{
    AUTHOR = {Schmidt, Wolfgang M.},
     TITLE = {\"{U}ber die {N}ormalit\"at von {Z}ahlen zu verschiedenen
              {B}asen},
   JOURNAL = {Acta Arith.},
  FJOURNAL = {Polska Akademia Nauk. Instytut Matematyczny. Acta Arithmetica},
    VOLUME = {7},
      YEAR = {1961/1962},
     PAGES = {299--309},
      ISSN = {0065-1036},
   MRCLASS = {10.33},
  MRNUMBER = {0140482 (25 \#3902)},
MRREVIEWER = {N. G. de Bruijn},
}

\bib{schmidt1972:irreg_of_distr_7}{article}{
    AUTHOR = {Schmidt, Wolfgang M.},
     TITLE = {Irregularities of distribution. {VII}},
   JOURNAL = {Acta Arith.},
  FJOURNAL = {Polska Akademia Nauk. Instytut Matematyczny. Acta Arithmetica},
    VOLUME = {21},
      YEAR = {1972},
     PAGES = {45--50},
      ISSN = {0065-1036},
   MRCLASS = {10K30 (10K25)},
  MRNUMBER = {0319933 (47 \#8474)},
MRREVIEWER = {I. Niven},
}

\bib{siegel:towards_a_usable_theory_of_chernoff}{article}{
    author = {Siegel, Alan},
    title = {Toward a usable theory of Chernoff Bounds for heterogeneous and partially dependent random variables},
    year = {1992}
}

\bib{sierpinski1917:borel_elementaire}{article}{
	    AUTHOR = {Sierpinski, Waclaw},
     TITLE = {D\'emonstration \'el\'ementaire du th\'eor\`eme de M. Borel sur les nombres absolument normaux et d\'etermination effective d'un tel nombre},
   JOURNAL = {Bulletin de la Soci\'et\'e Math\'ematique de France},
    VOLUME = {45},
      YEAR = {1917},
     PAGES = {127--132},
}

\bib{turing1992:collected_works}{article}{
    AUTHOR = {Turing, Alan},
     TITLE = {A Note on Normal Numbers},
     BOOKTITLE = {Collected Works of A. M. Turing, Pure Mathematics, edited by J. L. Britton},
   PUBLISHER = {North Holland},
      YEAR = {1992},
     PAGES = {117--119},
}

\bib{wall1950:thesis}{article}{
    AUTHOR = {Wall, Donald D.},
     TITLE = {Normal Numbers},
      NOTE = {Thesis (Ph.D.)--University of California, Berkeley},
 PUBLISHER = {ProQuest LLC, Ann Arbor, MI},
      YEAR = {1950},
   MRCLASS = {Thesis},
  MRNUMBER = {2937990},
       URL =
              {http://gateway.proquest.com/openurl?url_ver=Z39.88-2004&rft_val_fmt=info:ofi/fmt:kev:mtx:dissertation&res_dat=xri:pqdiss&rft_dat=xri:pqdiss:0169487},
}







\end{biblist}
\end{bibdiv}


\end{document}